\RequirePackage{etex}
\documentclass{amsart}
\usepackage{amsmath}
\usepackage{amsthm}
\usepackage{enumerate}
\usepackage{color}
\usepackage{float}
\usepackage{graphicx}
\usepackage{amssymb}
\usepackage{latexsym}
\usepackage{amsfonts}
\usepackage{cite}
\usepackage{hyperref}
\usepackage{mathtools}
\usepackage[all]{xy}
\usepackage{tikz, tikz-3dplot, tkz-graph}
\usepackage{tikz-cd}
\usetikzlibrary{decorations.markings}
\tikzstyle{vertex}=[circle, draw, inner sep=0pt, minimum size=13pt]

\allowdisplaybreaks

\hypersetup{unicode=false, pdftoolbar=true, pdfmenubar=true,
        pdffitwindow=false, pdfstartview={FitH}, pdfnewwindow=true,
colorlinks=false, linkcolor=blue, citecolor=green}

\DeclarePairedDelimiter\floor{\lfloor}{\rfloor}

\newtheorem{theorem}{Theorem}[section]
\newtheorem*{maintheorem}{Main Theorem}

\newtheorem{proposition}[theorem]{Proposition}

\newtheorem{corollary}[theorem]{Corollary}
\newtheorem{example}[theorem]{Example}
\newtheorem{lemma}[theorem]{Lemma}
\theoremstyle{definition}

\newtheorem{definition/notation}[theorem]{Definition/Notation}

\numberwithin{equation}{section}

\begin{document}

\title[Subprime Solutions of the CYBE]{Subprime Solutions of the Classical
Yang-Baxter Equation}

\author{Garrett Johnson}

\address{Department of Mathematics and Physics \\North Carolina Central
University\\Durham, NC 27707\\USA}

\email{gjohns62@nccu.edu}


\begin{abstract}

    We introduce a new family of classical $r$-matrices for the Lie algebra
    $\mathfrak{sl}_n$ that lies in the Zariski boundary of the Belavin-Drinfeld
    space ${\mathcal M}$ of quasi-triangular solutions to the classical
    Yang-Baxter equation.  In this setting ${\mathcal M}$ is a finite disjoint
    union of components; exactly $\phi(n)$ of these components are
    $SL_n$-orbits of single points. These points are the generalized
    Cremmer-Gervais $r$-matrices $r_{i, n}$ which are naturally indexed by
    pairs of positive coprime integers, $i$ and $n$, with $i < n$. A conjecture
    of Gerstenhaber and Giaquinto states that the boundaries of the
    Cremmer-Gervais components contain $r$-matrices having maximal parabolic
    subalgebras $\mathfrak{p}_{i,n}\subseteq \mathfrak{sl}_n$ as carriers. We
    prove this conjecture in the cases when $n\equiv \pm 1$ (mod $i$).  The
    subprime linear functionals $f\in\mathfrak{p}_{i, n}^*$ and the
    corresponding principal elements $H\in\mathfrak{p}_{i, n}$ play important
    roles in our proof. Since the subprime functionals are Frobenius precisely
    in the cases when $n\equiv \pm 1$ (mod $i$), this partly explains our need
    to require these conditions on $i$ and $n$.  We conclude with a proof of
    the GG boundary conjecture in an unrelated case, namely when  $(i, n) = (5,
    12)$, where the subprime functional is no longer a Frobenius functional.

\end{abstract}

\maketitle

\section{Introduction and Main Results}

Throughout this paper, we assume the ground field $\mathbb{F}$ has
characteristic $0$ although most results in this paper also hold for fields of
nearly any other characteristic. In particular, we will fix a pair of positive
integers $i$ and $n$ with $i < n$, and in some of the calculations that follow,
the numbers $2$ and $n$ appear in denominators.  Thus we need $2n$ to be a
nonzero element of the ground field $\mathbb{F}$. For vectors $u,v$ in an
$\mathbb{F}$-vector space $V$, define $u\wedge v := \frac{1}{2}\left(u \otimes
v - v\otimes u\right) \in V\wedge V \subseteq V\otimes V$. For a Lie algebra
$\mathfrak{g}$ and an element in the tensor space $r =\sum a_i\wedge b_i \in
\mathfrak{g} \wedge \mathfrak{g}$ we say that $r$ is a classical $r$-matrix if
the Schouten bracket of $r$ with itself

\begin{equation}
    \langle r, r\rangle : =  [r_{12}, r_{13}] + [r_{12}, r_{23}] + [r_{13},
    r_{23}]
\end{equation}

\noindent is $\mathfrak{g}$-invariant. Here, $r_{12} = r\otimes 1$, $r_{23} =
1\otimes r$, and $r_{13} = \sigma(r_{23})$, where $\sigma$ is the linear
endomorphism of $\mathfrak{g}\otimes\mathfrak{g}\otimes\mathfrak{g}$ that
permutes the first two tensor components: $\sigma(x\otimes y \otimes z) =
y\otimes x\otimes z$. Classical $r$-matrices arise naturally in the context of
Poisson-Lie groups and Lie bialgebras (see e.g. \cite[Chapter 1]{CP}). If
$\langle r, r\rangle = 0$, then $r$ is said to be a solution to the classical
Yang-Baxter equation (CYBE). On the other hand, if $\langle r, r\rangle$ is
non-zero and $\mathfrak{g}$-invariant, $r$ is said to be a solution to the
modified classical Yang-Baxter equation (MCYBE). Following \cite{GG}, we let
${\mathcal C}$ and ${\mathcal M}$ denote the solution spaces of the CYBE and
MCYBE respectively.

In the early 1980's Belavin and Drinfeld \cite{BD} classified the solutions to
the MCYBE for the finite-dimensional complex simple Lie algebras and showed
that the solution space ${\mathcal M}$ is a finite disjoint union of components
of the projective space $\mathbb{P}(\mathfrak{g} \wedge \mathfrak{g})$. The
components of ${\mathcal M}$ are indexed by triples ${\mathcal T} = ({\mathcal
T}, {\mathcal S}_1, {\mathcal S}_2)$, where ${\mathcal S}_1$ and ${\mathcal
S}_2$ are subsets of the set of simple roots of $\mathfrak{g}$ and ${\mathcal
T}: {\mathcal S}_1 \to {\mathcal S}_2$ is a bijection that preserves the
Killing form and satifies a nilpotency condition.  For any BD-triple
$({\mathcal T}, {\mathcal S}_1, {\mathcal S}_2)$, one can always produce
another BD-triple ${\mathcal T}^\prime$ by restricting ${\mathcal T}$ to a
subset of ${\mathcal S}_1$. This gives rise to the notion of a partial ordering
on triples: ${\mathcal T}^\prime < {\mathcal T}$.  An interesting family of
solutions of the MYCBE arises when considering maximal BD-triples with
${\mathcal S}_2$ missing a single root.  These occur only in the case when
$\mathfrak{g} = \mathfrak{sl}_n$, and in this setting there are exactly
$\phi(n)$ BD-triples of this type, where $\phi$ is the Euler-totient function
\cite{GG}. The component of ${\mathcal M}$ corresponding to such a triple can
each be described as the $SL_n$-orbit of a single point $r\in \mathfrak{sl}_n
\wedge \mathfrak{sl}_n$, called the \textit{Cremmer-Gervais $r$-matrix}
\cite{CG, EH}.  Hence, throughout we let $i$ and $n$ be a pair of positive
coprime integers with $i < n$ and let $r_{CG}(i, n)$ denote the corresponding
Cremmer-Gervais $r$-matrix of type $(i, n)$. The BD-triple associated to
$r_{CG}(i, n)$ has the $i$-th simple root missing from ${\mathcal S}_2$. The
Cremmer-Gervais $r$-matrices have explicit formulas which we describe in
Section \ref{CG r-matrices}.  For example when $i = 1$ and $n = 3$, we have

\begin{align*}
    r_{CG}(1, 3) &= 2e_{12}\wedge e_{32} + e_{12}\wedge e_{21} +
    e_{13}\wedge e_{31} + e_{23} \wedge e_{32}
    \\
    &\hspace{14mm}+\frac{1}{3}\left(e_{11}-e_{22}\right) \wedge \left( e_{22}
    - e_{33} \right)\in \mathfrak{sl}_3 \wedge \mathfrak{sl}_3.
\end{align*}

A classification of solutions to the CYBE, on the other hand, is quite
difficult and would entail classifying all quasi-Frobenius Lie algebras
\cite{Stolin}. However, in certain cases there is a straightforward technique
for constructing non-trivial solutions to the CYBE. For instance, if a Lie
algebra $\mathfrak{f}$ has a functional $f\in\mathfrak{f}^*$ such that $x
\mapsto f([x, -])$ is an isomorphism from $\mathfrak{f}$ to $\mathfrak{f}^*$,
then, for an ordered basis $x_1, x_2, \dots, x_d$ of $\mathfrak{f}$, the
bilinear form $\langle x_j, x_k\rangle := f([x_j,x_k])$ on $\mathfrak{f}$ is
invertible and $\sum (f[x_j,x_k])^{-1}x_j\wedge x_k$ produces a non-degenerate
solution to the CYBE. In such a setting, we call $\mathfrak{f}$ a
\textit{Frobenius} Lie algebra and $f$ is called a \textit{Frobenius}
functional. The inverse image of $f$ under the map $x\mapsto f([x, -])$ is
called the \textit{principal element} \cite{GG3}. The $r$-matrix produced from
this construction is said to have \textit{carrier} $\mathfrak{f}$.

Frobenius algebras arise in several other contexts.  For instance, in \cite{O},
Ooms shows that a universal enveloping algebra ${\mathcal U}(L)$ of a Lie
algebra $L$ over a field $\mathbb{F}$ of characteristic $0$ is primitive if and
only if $L$ is Frobenius.  Some interesting families of Frobenius subalgebras
of $\mathfrak{sl}_n$ include certain types of parabolic and biparabolic Lie
algebras (see e.g. \cite{CGM, DK, P}).  In this paper we focus on the maximal
parabolic subalgebras $\mathfrak{p}(i, n) \subseteq \mathfrak{sl}_n$, where the
$i$-th negative root vector is deleted. In \cite{E}, Elashvili showed that
$\mathfrak{p}(i, n)$ is Frobenius if and only if $i$ is relatively prime to
$n$. Hence, if $f\in \mathfrak{p}(i, n)^*$ is a Frobenius functional, then a
solution $r_f\in \mathfrak{sl}_n \wedge \mathfrak{sl}_n$ to the CYBE having
carrier subalgebra $\mathfrak{p}(i, n)$ can be constructed from $f$.

In \cite{GG}, it was shown that the Zariski boundary of ${\mathcal M}$ is
contained in ${\mathcal C}$.  Thus, one would hope that there might be a simple
description of boundary solutions analogous to the BD-classification, but this
seems to be a very difficult problem. In an effort to better understand the
boundary of ${\mathcal M}$, Gerstenhaber and Giaquinto conjectured that the
Zariski boundary of the component of ${\mathcal M}$ containing $r_{CG}(i, n)$
contains an $r$-matrix having parabolic carrier $\mathfrak{p}(i, n)$. They
prove their conjecture in the case when $i=1$ \cite{GG}. In Section \ref{main
result section}, we extend these results by proving their conjecture also holds
in the cases when $n \equiv \pm 1 $ (mod $i$).  As it turns out, the subprime
functional

\[
    f_{subprime} := \sum_{i < j\leq n} e_{j-i,j}^* + \sum_{1\leq j < i} e_{j+1,
    j}^* \in \mathfrak{p}(i, n)^*
\]

\noindent is Frobenius if and only if $n\equiv \pm 1$ (mod $i$) (see e.g.
\cite[Section 9.1]{GG2}).  Due to the key role that the subprime functional
plays in our proofs, we refer to the cases when $n \equiv \pm 1$ (mod $i$) as
\textit{subprime} cases.

On the way towards proving that the boundary conjecture holds in the subprime
cases, we define in Section \ref{subsection phi} a rational map $\Phi:
\mathbb{F}^\times \to SL_n$ (that depends on $i$ and $n$), where
$\mathbb{F}^\times = \mathbb{F}\backslash\{0\}$.  By rational, we mean the
matrix entries in $\Phi(t)$ are polynomials in $t^{\pm 1}$. The principal
element associated to the subprime functional is needed in the construction of
$\Phi$, and, in fact, all entries of $\Phi(t)$ have the form $ct^p$ with $c$
and $p$ integers. For example, when $i = 3$ and $n = 7$ our formula for $\Phi$
is given by

\[
    \Phi: t\mapsto
    \begin{bmatrix}
        t^{-2} & 0 & 0 & -2t^{-2} & 0 & 0 & t^{-2} \\
        -2t^{12} & t^{12} & 0 & t^{12} & -t^{12} & 0 & t^{12} \\
        t^{26} & -t^{26} & t^{26} & -2t^{26} & t^{26} & -t^{26} &t^{26} \\
        0 & 0 & 0 & t^{-16} & 0 & 0 & -t^{-16} \\
        0 & 0 & 0 & -2t^{-2} & t^{-2} & 0 & -t^{-2} \\
        0 & 0 & 0 & t^{12} & -t^{12} & t^{12} & -t^{12} \\
        0 & 0 & 0 & 0 & 0 & 0 & t^{-30}
    \end{bmatrix}
    \in SL_n.
\]

We use the map $\Phi$ to produce a one-parameter family of solutions to the
MYCBE lying in $SL_n$-orbit of the Cremmer-Gervais $r$-matrix. More precisely,
we show that $\Phi(t).r_{CG}(i, n)$ is of the form $r^\prime + t^{2n}b$, where
$r^\prime\in{\mathcal M}$ and $b\in{\mathcal C}$. Here $\Phi(t)\in SL_n$ acts
on $r_{CG}(i, n)$ via the adjoint action.  In a sense the map $\Phi$ is used to
deform the Cremmer-Gervais $r$-matrix, and $t$ can be viewed as a deformation
parameter. Under mild conditions (which are met in our case) it follows that
the coefficient of the highest degree term in $t$ is a boundary solution to the
CYBE \cite[Proposition 5.1]{GG}. Hence, $b$ is a boundary solution, which we
refer to as a \textit{subprime} boundary solution of the classical Yang-Baxter
equation. Our main result is the following theorem that proves the
Gerstenhaber-Giaquinto boundary conjecture holds in the subprime cases.

\begin{maintheorem}

        Let $i$ and $n$ be a pair of positive integers with $i<n$ so that
        $n\equiv\pm 1$ (mod $i$) and let $r = r_{CG}(i, n)$ denote the
        Cremmer-Gervais $r$-matrix of type $(i, n)$.  Then

        \begin{enumerate}

            \item $\Phi(t).r$ is of the form $r^\prime + t^{2n}b$, where
                $r^\prime\in{\mathcal M}$ and $b\in{\mathcal C}$,

            \item the $r$-matrices $r^\prime$ and $b$ appearing in part (1) lie
                in the Zariski closure of the component of ${\mathcal M}$
                containing $r$.

        \item the $r$-matrix $b$ is a boundary solution to the classical
                Yang-Baxter equation,

            \item the carrier of $b$ is the maximal parabolic subalgebra
                $\mathfrak{p}(i, n) \subseteq \mathfrak{sl}_n$.

        \end{enumerate}

\end{maintheorem}

In Section \ref{main result section}, we give explicit formulas for the
$r$-matrices $r^\prime$ and $b$ appearing in the statement of our main result.
Interestingly, the principal element $H$ associated to the subprime linear
functional plays an important role in our construction of $\Phi$. In
Proposition \ref{H_M} we prove a curious relationship between $r^\prime$ and
$H$, specifically that $r^\prime$ is in the kernel of the adjoint action of
$H$.  Since the subprime functional is a Frobenius functional in
$\mathfrak{p}(i, n)^*$ if and only if $n \equiv \pm 1$ (mod $i$), this partly
explains the need for us to assume these conditions on $i$ and $n$ in our
proof.

In Section \ref{closing}, we conclude with a proof of the
Gerstenhaber-Giaquinto boundary conjecture in a case unrelated to the subprime
cases, namely when $i = 5$ and $n = 12$, where the subprime functional is no
longer Frobenius.

\vspace{3mm}

\noindent \textit{Acknowledgements:} The author was partially supported by the
United States Department of Homeland Security grant DHS-16-ST-062-000004.

\section{Solutions of the MCYBE}

\subsection{Belavin-Drinfeld Classification}

\label{BD classification}

Solutions to the modified classical Yang-Baxter equation (MCYBE) for a finite
dimensional complex simple Lie algebra $\mathfrak{g}$ were constructively
classified by Belavin and Drinfeld \cite{BD}.  Their classification result says
that the solution space ${\mathcal M}$ is a finite disjoint union of
quasi-projective subvarieties of $\mathbb{P}(\mathfrak{g} \wedge
\mathfrak{g})$.  The components of ${\mathcal M}$ are indexed by combinatorial
objects, called BD-triples, on the Dynkin diagram of $\mathfrak{g}$. A
BD-triple  ${\mathcal T}$ puts a particular partial ordering $\preceq$ on the
positive root vectors of $\mathfrak{g}$. This ordering is used to define
$\alpha = 2\sum_{\rho\prec\mu} e_\rho \wedge e_\mu
\in\mathfrak{g}\wedge\mathfrak{g}$ and a subvariety $\beta({\mathcal T})$ of
$\mathfrak{h} \wedge \mathfrak{h}$, where $\mathfrak{h}$ is a Cartan subalgebra
of $\mathfrak{g}$. The BD classification result states that each $r\in{\mathcal
M}$ is equivalent, up to scaling and applying inner automorphisms of
$\mathfrak{g}$, to a unique $r$-matrix of the form $\alpha + \beta + \gamma$,
where $\beta\in \beta({\mathcal T})$ and $\gamma$ is a certain fixed element of
$\mathfrak{g} \wedge \mathfrak{g}$ independent of ${\mathcal T}$. For further
details on the BD-classification, see e.g. \cite[Chapter 1]{CP}.

For our purposes, rather than give a full description of the Belavin-Drinfeld
classification, we turn our attention to the case when $\mathfrak{g} =
\mathfrak{sl}_n$, the Lie algebra of traceless $n\times n$ matrices.  Let
$\mathfrak{g} = \mathfrak{n}^-\oplus \mathfrak{h} \oplus \mathfrak{n}^+$ denote
the triangular decomposition. In this setting, a BD-triple ${\mathcal T} =
({\mathcal T}, {\mathcal S}_1, {\mathcal S}_2)$ is a bijection ${\mathcal T}:
{\mathcal S}_1\to{\mathcal S}_2$, where ${\mathcal S}_1$ and ${\mathcal S}_2$
are subsets of $[1, n-1]$  (here and below we use the notation
$[a, b]$ for $\{a,a+1,\dots, b\}$), satisfying the
conditions

\begin{enumerate}

    \item (adjacency preserving) $|{\mathcal T}(j) - {\mathcal T}(k)| = 1$ if
        and only if $|j - k| = 1$,

    \item (local nilpotency) for every $j\in{\mathcal S}_1$ there exists
        $N\in\mathbb{N}$ so that ${\mathcal T}^N(j)\notin{\mathcal S}_1$.

\end{enumerate}

\noindent A BD-triple defines a partial ordering $\prec$ on $[1, n-1]$
by $j\prec k$ if and only if there exists $N\in\mathbb{N}$ so that ${\mathcal
T}^N(j) = k$. This ordering extends to the set $\{e_{jk} \in \mathfrak{sl}_n
\mid 1\leq j<k\leq n\}$ of positive root vectors of $\mathfrak{sl}_n$: put
$e_{jk}\prec e_{\ell m}$ if and only if $m = k-j+\ell$ and $j + s \prec \ell +
s $ for all $s\in[0,k-j-1]$. We define

\begin{equation}
    \alpha := 2\sum e_{jk}\wedge e_{m\ell} \in \mathfrak{n}^+ \wedge
    \mathfrak{n}^-,
\end{equation}

\noindent where the sum is over all tuples $(j,k,\ell,m)\in[1,n]^4$
such that $e_{jk} \prec e_{\ell m}$.  Next let $h_j := e_{j,j} - e_{j+1,j+1}$
for $j\in[1, n - 1]$ and let $e_{jk}^*\in\mathfrak{g}^*$ be the linear
functional that returns the $(j,k)$-entry.  Next define the subvariety
$\beta({\mathcal T}) \subseteq \mathfrak{h} \wedge \mathfrak{h}$ by

\begin{equation}
    \beta({\mathcal T}) :=
    \Big\{\beta\in\mathfrak{h}\wedge\mathfrak{h} \mid (1 \otimes (e_{{\mathcal
    T}(j), {\mathcal T}(j) + 1}^* - e_{j, j + 1}^*))\beta = \frac{1}{2}(
    h_{{\mathcal T}(j)} + h_j)\text{ for all } j \in {\mathcal S}_1\Big\}.
\end{equation}

\noindent As a variety $\beta({\mathcal T})$ has dimension $d(d-1)/2$, where $d
= n - 1 - |{\mathcal S}_1|$.  Finally define

\begin{equation}
    \gamma := \sum_{1\leq k < \ell\leq n} e_{k\ell} \wedge e_{\ell k} \in
    \mathfrak{n}^+ \wedge \mathfrak{n}^-.
\end{equation}

Observe that the solution space ${\mathcal M}$ is stable under the adjoint
action of the special linear group $SL_n$ and also by rescaling by any nonzero
scalar $\lambda$. We call $r, r^\prime \in {\mathcal M}$ \textit{equivalent}
$r$-matrices if there exists a nonzero scalar $\lambda$ and $g\in SL_n$ so that
$\lambda g.r=r^\prime$.  Belavin and Drinfeld's classification result
(specialized to the setting $\mathfrak{g} = \mathfrak{sl}_n$) asserts that $r =
\alpha + \beta + \gamma \in \mathfrak{sl}_n \wedge \mathfrak{sl}_n$ is a
solution to the MCYBE for every BD-triple ${\mathcal T} = ({\mathcal T},
{\mathcal S}_1, {\mathcal S}_2$) and $\beta\in \beta(\mathcal{T})$, and
conversely any solution to the MCYBE is equivalent to a unique $r$-matrix of
this form.

\subsection{Cremmer-Gervais \texorpdfstring{$r$-matrices}{r-matrices}}

\label{CG r-matrices}

We turn our attention to a particularly interesting family of $r$-matrices
called Cremmer-Gervais $r$-matrices (see \cite{CG, EH}). These are precisely
the solutions to the MCYBE for the Lie algebra $\mathfrak{g} = \mathfrak{sl}_n$
associated to \textit{maximal} BD-triples, i.e. those with ${\mathcal S}_1$
having the maximum possible cardinality, $\#{\mathcal S}_1 = n - 2$.

Our aim in this section is to explicitly describe the Cremmer-Gervais
$r$-matrices.  First of all there are exactly $\phi(n)$ maximal BD-triples,
where $\phi$ is the Euler-totient function. Hence, throughout the rest of this
section we fix a pair of positive coprime integers $i$ and $n$ with $i < n$ and
let $r_{CG}(i, n)\in{\mathcal M}$ denote the Cremmer-Gervais $r$-matrix of type
$(i, n)$.  We define a bijection ${\mathcal T}_{i, n}$ from the set ${\mathcal
S}_1 := [1,n-i-1]\cup [n-i+1,n-1]$ to the set ${\mathcal S}_2
:= [1,i-1] \cup [i+1, n-1]$ by ${\mathcal T}_{i, n}(j) = (j +
i) \text{ mod } n$ and construct a directed graph $\Gamma_{CG}(i, n)$ having
vertices labelled $1, 2, \dots, n-1$ and edges $j\to{\mathcal T}_{i, n}(j)$ for
every $j\in{\mathcal S}_1$. The graph $\Gamma_{CG}(i, n)$ is, in fact,
isomorphic to $i\to 2i \to \cdots \to n - i$, where the numbers are assumed to
be reduced modulo $n$.

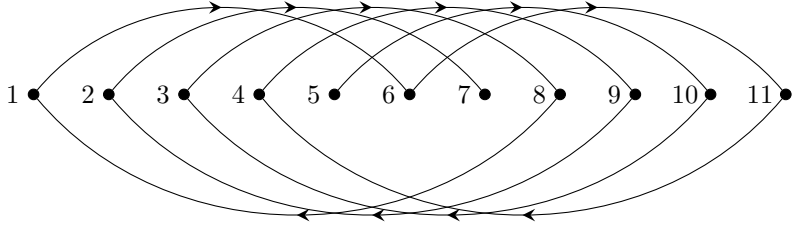
\begin{figure}[h!]
\[
\begin{tikzpicture}
    \tikzstyle{every node}=[draw,circle,fill=black,minimum size=4pt,
    inner sep=1pt]
    \tikzset{->-/.style={decoration={markings,
    mark=at position #1 with {\arrow[scale=1.5,>=stealth]{>}}},
    postaction={decorate}}}
    \draw (0,0) node[label=left:$1$] (0) {};
    \draw (1,0) node[label=left:$2$] (1) {};
    \draw (2,0) node[label=left:$3$] (2) {};
    \draw (3,0) node[label=left:$4$] (3) {};
    \draw (4,0) node[label=left:$5$] (4) {};
    \draw (5,0) node[label=left:$6$] (5) {};
    \draw (6,0) node[label=left:$7$] (6) {};
    \draw (7,0) node[label=left:$8$] (7) {};
    \draw (8,0) node[label=left:$9$] (8) {};
    \draw (9,0) node[label=left:$10$] (9) {};
    \draw (10,0) node[label=left:$11$] (10) {};
    \draw [->-=.5] (0) to [bend left=50] (5);
    \draw [->-=.5] (1) to [bend left=50] (6);
    \draw [->-=.5] (2) to [bend left=50] (7);
    \draw [->-=.5] (3) to [bend left=50] (8);
    \draw [->-=.5] (4) to [bend left=50] (9);
    \draw [->-=.5] (5) to [bend left=50] (10);
    \draw [->-=.5] (7) to [bend left=50] (0);
    \draw [->-=.5] (8) to [bend left=50] (1);
    \draw [->-=.5] (9) to [bend left=50] (2);
    \draw [->-=.5] (10) to [bend left=50] (3);
\end{tikzpicture}
\]

\caption{The Cremmer - Gervais graph $\Gamma_{CG}(5, 12)$}
\label{CG graph example}
\end{figure}

\noindent We will refer to the edges $j\to k$ with $j<k$ as the forward arrows,
whereas the edges $j\to k$ with $j > k$ are called the backward arrows. In
Figure \ref{CG graph example}, we draw the forward arrows in the top half of
the graph and the backward arrows in the bottom half.  Recall the map
${\mathcal T}_{i, n}$ defines an ordering $\prec$ on $[1, n-1]$.
Equivalently $j\prec k$ if and only if the path in $\Gamma_{CG}(i, n)$ that
starts at vertex $j$ eventually reaches vertex $k$ by following the directed
edges.  In fact, the ordering is simply $i \prec 2i \prec 3i \prec \dots \prec
n - i$, where the numbers are assumed to be reduced modulo $n$.

Let $\alpha_{\rightarrow}$ denote the $\alpha$-part of $r_{CG}(i, n)$ obtained
by considering only the forward arrows, and let $\alpha_{\leftarrow}$ denote
the terms introduced to $\alpha$ by including the backward arrows. Thus,
$\alpha = \alpha_{\rightarrow} +\alpha_{\leftarrow}$.  To describe the
$\alpha$-part, we find it convenient to first define the following elements of
the general linear Lie algebra $\mathfrak{gl}_n$:

\begin{align}
    \xi_{k\ell} := \sum_{p\in\mathbb{Z}} e_{k+ip,\ell+ip} \in\mathfrak{gl}_n,
    & &
    \eta_{k\ell} := \sum_{p\geq 0} e_{k+ip,\ell+ip} \in\mathfrak{gl}_n,
\end{align}

\noindent for all $k, \ell \in\mathbb{Z}$.  Throughout, we treat
$\mathfrak{sl}_n$ as a subset of $\mathfrak{gl}_n$ and follow the convention
$e_{jk} := 0$ if either subscript is out of range, i.e. if either $j$ or $k$ is
greater than $n$ or less than $1$.  For integers $p,q\in\mathbb{Z}$ with $q>0$,
we write $p \text{ mod }q$ to denote the reduction of $p$ modulo $q$ in
$[0, q-1]$ .  The $\alpha$-parts of the Cremmer-Gervais $r$-matrix
$r_{CG}(i, n)$ are

\begin{equation}
    \label{alpha_forward}
    \alpha_{\rightarrow} = 2\hspace{-4mm}\sum_{1\leq k<\ell\leq n-i}
    \hspace{-4mm} e_{k\ell}\wedge \eta_{\ell+i,k+i} \in \mathfrak{n}^+ \wedge
    \mathfrak{n}^-
\end{equation}

\noindent and

\begin{equation}
    \label{alpha_backward}
    \alpha_{\leftarrow} = 2\sum \xi_{j + (n \text{ mod }i), k + (n
    \text{ mod }i)}\wedge \xi_{m \ell} \in \mathfrak{n}^+ \wedge
    \mathfrak{n}^-,
\end{equation}

\noindent where the sum above runs over all tuples $(j,k,\ell, m) \in [1, i]^4$
such that $e_{jk} \preceq e_{\ell m}$ in the partial ordering on the positive
root vectors of $\mathfrak{sl}_i$ (in particular, $1\leq j < k\leq i$ and
$1\leq \ell < m \leq i$) induced by the Cremmer-Gervais graph $\Gamma_{CG}(i -
(n \text{ mod } i), i)$. Equation \ref{alpha_backward} provides a way to
iteratively construct $\alpha_{\leftarrow}$. The number of steps needed to
construct $\alpha_{\leftarrow}$ equals the number of steps until the Euclidean
algorithm, starting with the integers $n$ and $i$, terminates.

The varieties $\beta({\mathcal T}_{i,n})$ associated to maximal BD-triples each
reduce to a single point. In \cite{GG} a formula for this point is given by

\begin{equation}
    \beta := \sum_{1\leq j < \ell\leq n}
    \left(-1+\frac{2}{n}\left((j-\ell)i^{-1}\text{ mod }n\right)\right)
    e_{jj}\wedge e_{\ell\ell} \in \mathfrak{h} \wedge \mathfrak{h}.
\end{equation}

\noindent The Cremmer-Gervais $r$-matrix $r_{CG}(i, n)$ is equal to $r_{CG}(i,
n) = \alpha_{\rightarrow} + \alpha_{\leftarrow} + \beta + \gamma \in {\mathcal
M}$.

\section{Solutions of the CYBE}

\subsection{Quasi-Frobenius and Frobenius subalgebras}

\label{section_CYBEsolutions}

The solution space ${\mathcal C}$ to the classical Yang-Baxter equation (CYBE),
on the other hand, is difficult to describe as there is not a constructive
classification analogous to the classification result of Belavin and Drinfeld
for the solution space ${\mathcal M}$. However we recall a homological
description of solutions to the CYBE \cite{Stolin}.

For $r\in{\mathcal C}$, let $\mathfrak{f}\subseteq\mathfrak{g}$ be the
\textit{carrier} of $r$. The carrier is the Lie subalgebra of $\mathfrak{g}$
spanned by $\{(\xi\otimes 1)r \mid \xi \in \mathfrak{g}^*\}$. The map
$\check{r}:\mathfrak{f}^*\to \mathfrak{f}$ defined by $\xi\mapsto (\xi\otimes
1)r$ is a linear isomorphism and induces a Lie algebra $2$-cocycle $B:
\mathfrak{f}\times\mathfrak{f}\to\mathbb{F}$ given by $(x, y)\mapsto
\langle\check{r}^{-1}(x), y\rangle$. Thus, $\mathfrak{f}$ is a quasi-Frobenius
Lie algebra. This process be can inverted to give a one-to-one correspondence
between quasi-Frobenius Lie algebras $(\mathfrak{f}, B)$ and $r$-matrices in
${\mathcal C}$ having carrier $\mathfrak{f}$. Thus classifying all solutions to
the CYBE would entail classifying all quasi-Frobenius subalgebras
$\mathfrak{f}\subseteq\mathfrak{g}$.

If the cocycle $B$ corresponding to $r$ is a coboundary, this means that $B$ is
of the form $B(x, y) = f([x, y])$ for some linear functional
$f\in\mathfrak{f}^*$. In such cases we say that $r$ \textit{admits} the
functional $f$ and we call $\mathfrak{f}$ a \textit{Frobenius Lie algebra} and
call $f$ a \textit{Frobenius functional}.  Here, the linear map $x\mapsto
f([x,-])$ is an isomorphism from $\mathfrak{f}$ to its dual space
$\mathfrak{f}^*$.  The inverse image of $f$ under this map is called the
\textit{principal element}, which is the unique $H\in\mathfrak{f}$ satisfying
$f([H,x]) = f(x)$ for all $x\in\mathfrak{f}$ (see e.g. \cite{GG3}).

\begin{example}

    The maximal parabolic subalgebra $\mathfrak{p}(i, n) \subseteq
    \mathfrak{sl}_n$ obtained by deleting the $i$-th negative simple root is
    Frobenius if and only if $i$ and $n$ are relatively prime \cite{E} (see
    Figure \ref{maximal parabolic figure}).

    \begin{figure}[h!]

        \[
            \mathfrak{p}(2, 5) =
                \begin{bmatrix}
                    * & * & * & * & * \\
                    * & * & * & * & * \\
                    0 & 0 & * & * & * \\
                    0 & 0 & * & * & * \\
                    0 & 0 & * & * & * \\
                \end{bmatrix}
            \hspace{10mm}
            \mathfrak{p}(4, 9) =
                \begin{bmatrix}
                    * & * & * & * & * & * & * & * & *\\
                    * & * & * & * & * & * & * & * & *\\
                    * & * & * & * & * & * & * & * & *\\
                    * & * & * & * & * & * & * & * & *\\
                    0 & 0 & 0 & 0 & * & * & * & * & *\\
                    0 & 0 & 0 & 0 & * & * & * & * & *\\
                    0 & 0 & 0 & 0 & * & * & * & * & *\\
                    0 & 0 & 0 & 0 & * & * & * & * & *\\
                    0 & 0 & 0 & 0 & * & * & * & * & *\\
                \end{bmatrix}
        \]

    \caption{Maximal parabolic Frobenius Lie algebras $\mathfrak{p}(2, 5)$ and
    $\mathfrak{p}(4, 9)$}
    \label{maximal parabolic figure}
    \end{figure}

\end{example}

\begin{example}

    \label{principal_elt_subprime}

    The subprime functional $f = \displaystyle{\sum_{i < j\leq n} e_{j-i,j}^* +
    \sum_{1\leq j < i} e_{j+1, j}^* \in \mathfrak{p}(i, n)^*}$ is Frobenius if
    and only if $n\equiv \pm 1$ (mod $i$) (see e.g. \cite[Section 9.1]{GG2}).
    The associated principal element is

    \begin{equation}
        H = \operatorname{diag}(0, 1, 2, \dots, i-1, -1, 0, 1, \dots, i-2, -2,
        -1, 0, \dots) + \Theta {\bf I}_n \in \mathfrak{p}(i, n).
    \end{equation}

    \noindent where $\Theta = \frac{n-1}{2n} \left(\frac{n+1}{i} - i\right)$ is
    a scalar making $H$ traceless. When $i = 1$ the subprime functional reduces
    to the \textit{prime} functional and, up to a scalar multiple, $H$ reduces
    to the semisimple element of Konstant's principal three-dimensional
    subalgebra of $\mathfrak{sl}_n$ \cite{Kostant TDS}.

\end{example}

\subsection{Boundary Solutions of the CYBE}

In \cite{GG} Gerstenhaber and Giaquinto introduce boundary solutions to the
CYBE and show, among other things, that all points lying in the Zariski
boundary of ${\mathcal M}$ are contained in ${\mathcal C}$. Thus one would hope
that there might be a fairly simple description of boundary $r$-matrices
analogous to the Belavin-Drinfeld classification.  We will make use of the
following theorem and corollary found in \cite{GG}, which together provide a
general technique for constructing a large number of examples of boundary
$r$-matrices. In the following theorem and corollary the variable $t$ is
treated as a formal parameter; the base field $\mathbb{F}$ is enlarged to
$\mathbb{F}[t]$.

\begin{theorem}

    \label{constructing boundary solutions, 1}

    If $r\in{\mathcal M}$ and $r_t = r + r_1t + r_2t^2 + \cdots
    r_mt^m\in{\mathcal M}$ with $\langle r, r\rangle = \langle r_t,
    r_t\rangle$, then $r_m$ is a boundary solution to the CYBE.

\end{theorem}

\begin{corollary}

    \label{constructing boundary solutions, 2}

    Suppose $x\in\mathfrak{g}$ is nilpotent and $r\in{\mathcal M}$. Then
    $exp(tx)r$ has the form $exp(tx)r = r + tr_1 + \cdots + t^mr_m$ with
    $r_m\in {\mathcal C}$ a boundary solution to the CYBE.

\end{corollary}

\begin{example}

    \label{example, boundary solution}

    Let $e = e_{12} \in \mathfrak{sl}_2$, $f = e_{21} \in \mathfrak{sl}_2$, and
    $h = e_{11} - e_{22} \in \mathfrak{sl}_2$ be the standard basis of
    $\mathfrak{sl}_2$. We have $r := e\wedge f\in{\mathcal M}$ and $exp(t\cdot
    e)r = r + t(e\wedge h)$. Therefore, $e \wedge h\in{\mathcal C}$ is a
    boundary $r$-matrix.

\end{example}

\section{Towards a Proof of the Main Result}

We next introduce four objects that play important roles in our proof of the
Gerstenhaber-Giaquinto boundary conjecture in the \textit{subprime} cases,
namely the cases when $n \equiv \pm 1$ (mod $i$).  These objects are (1) a Lie
subalgebra $\mathfrak{n}$ of $\mathfrak{sl}_n$, (2) an $\mathfrak{n}$-module
$M$, (3) the principal element $H \in \mathfrak{p}(i, n)$ corresponding to the
subprime functional, and (4) a map $\Phi: \mathbb{F}^\times \to SL_n$.  After a
discussion of some key properties of $\mathfrak{n}$, $M$, $H$, and $\Phi$ we
will then be in a position to prove our main result.

\begin{itemize}

    \item From now on, unless stated otherwise, we assume $i$ and $n$ are a
        pair of positive integers with $i < n$ such that $n\equiv \pm 1$ (mod
        $i$).

\end{itemize}

\subsection{The Lie Algebra \texorpdfstring{$\mathbf{\mathfrak{n}}$}{n}}
\label{subsection n}

We define $\epsilon := 1$ if $n \text{ mod } i = 1$ and $\epsilon := -1$ if $n
\text{ mod } i \neq 1$. For $k\in\mathbb{Z}$ we let $c_k := \floor*{\frac{n -
k}{i}}$ and $a_k := \epsilon ((-nk) \text{ mod }i)$.  Next define the matrices

\begin{equation}
    X := \sum_{1\leq j\leq n - i}c_je_{j,j+i} \in \mathfrak{sl}_n,
    \hspace{10mm}
    Z := \sum_{1\leq j < i}a_j\xi_{j+1,j} \in \mathfrak{sl}_n.
\end{equation}

Let $\mathfrak{n}$ be the Lie subalgebra of $\mathfrak{sl}_n$ generated by the
matrices $X$ and $Z$.  The Lie algebra $\mathfrak{n}$ is a nilpotent Lie
subalgebra of $\mathfrak{sl}_n$. When $i=1$, $\mathfrak{n} = \mathbb{F}X$. For
$i > 1$ the matrices $Z, X, (\operatorname{ad}Z)X, (\operatorname{ad}Z)^2X,
\dots, (\operatorname{ad}Z)^{i-1}X$ form a linear $\mathbb{F}$-basis for
$\mathfrak{n}$, the $\mathbb{F}$-span of $X,(\operatorname{ad}Z)X,
(\operatorname{ad}Z)^2X, \dots, (\operatorname{ad}Z)^{i-1}X$ is an abelian Lie
subalgebra of $\mathfrak{n}$ of codimension 1, and $(\operatorname{ad}Z)^iX=0$.

\subsection{The \texorpdfstring{$\mathbf{\mathfrak{n}}$-module
$\mathbf{M}$}{n-module M}}

\label{subsection M}

Let $r := r_{CG}(i, n)$ be the Cremmer-Gervais $r$-matrix of type $(i, n)$ and
let $M$ be the $\mathfrak{n}$-module generated by $r$, i.e. $M = {\mathcal
U}(\mathfrak{n}).r \subseteq \mathfrak{sl}_n \wedge \mathfrak{sl}_n$ where
${\mathcal U}(\mathfrak{n})$ is the universal enveloping algebra of
$\mathfrak{n}$. Our aim is to explicitly describe the $\mathfrak{n}$-module $M$
by giving a list of basis vectors of $M$ and providing formulas for the actions
of $X$ and $Z$ on each basis vector. In computing these formulas we sometimes
find it convenient to write $\alpha_{\leftarrow}$ and $Z$ as

\begin{equation}
    \alpha_{\leftarrow} = 2\sum_{1\leq \ell \leq j <k \leq i}E_{jk}\wedge E_{k
    - j + \ell - 1, \ell - 1},
    \hspace{10mm}
    Z = \sum_{1\leq j<i}(i-j)E_{j,j-1}.
\end{equation}

\noindent where, for $k,\ell\in\mathbb{Z}$, $E_{k\ell} := \xi_{k + 1, \ell +
1}$ if $n \text{ mod }i = 1$ and $E_{kl} := -\xi_{i - \ell, i - k}$ if $n
\text{ mod } i \neq 1$.  Next define the following elements of $\mathfrak{p}(i,
n) \wedge \mathfrak{p}(i, n)$:

\begin{align}
    \label{V0}
    V_0 &:=
            \displaystyle{
            \frac{2}{n}X\wedge I-2\sum_{
                \bgroup
                    \def\arraystretch{0.5}
                    \begin{array}{c}
                        \scriptscriptstyle{1\leq k<\ell\leq n}
                        \\
                        \scriptscriptstyle{c_k > c_\ell}
                    \end{array}
                \egroup}
            e_{k\ell}\wedge \eta_{\ell,k+i}
            -2\sum_{1\leq k < \ell < i}E_{k\ell}\wedge E_{\ell - k,i}},
        \\
    \label{W}
    W &:=
        \displaystyle{2\left(\sum_{1\leq\ell<i} d_\ell\wedge
        E_{\ell,\ell - 1} + \sum_{1\leq\ell<j<k\leq i} E_{j-1,k-1}\wedge
        E_{k-j+\ell,\ell - 1}\right)},
    \\
    V_{\ell} &:=
    \displaystyle{-2\binom{i}{\ell}\left(d_\ell \wedge E_{\ell i} +
        \hspace{-2mm}\sum_{\ell\leq j<k<i} \hspace{-2mm} E_{jk} \wedge E_{k - j
        + \ell,i} - \sum_{0\leq k<j<\ell}E_{jk} \wedge E_{k - j + \ell,i}
        \right)}
    \end{align}

\noindent for $\ell\in[1,i]$, where $d_\ell$ is the diagonal matrix in
$\mathfrak{sl}_n$ defined by $d_\ell : = \sum_{\ell\leq j < i}E_{jj} +
\varTheta \mathbf{I}_n$.  The scalar $\varTheta\in\mathbb{F}$ is determined by
the condition that $d_\ell$ is traceless.  The following proposition
characterizes the $\mathfrak{n}$-module $M$.

\begin{proposition}

    \label{module_actions}

    $ $

    \begin{enumerate}

        \item \label{Zr} $Z.r = W$,

        \item \label{Xr} $X.r = V_0$,

        \item \label{ZW} $Z.W = 0$,

        \item \label{ZV} $Z.V_\ell = (\ell + 1)V_{\ell + 1}$ for all
            $\ell\in[0,i]$,

        \item \label{XW} $X.W = \frac{1}{i}V_1$, and

        \item \label{XV} $X.V_\ell = 0$

    \end{enumerate}

    with the convention that $V_{i + 1} = 0$.

\end{proposition}

When $i > 1$, the set of vectors $\{r, W, V_0, V_1, \dots V_{i}\}$ is an
$\mathbb{F}$-basis of $M$.  However if $i = 1$, the vectors $W$ and $V_1$
vanish; in this case $\{r, V_0\}$ is an $\mathbb{F}$-basis of $M$.  Before
proving Proposition \ref{module_actions} we provide two lemmas, both
computational in nature.

\begin{lemma}

    \label{lemma 1}

    $ $

    \begin{enumerate}

    \item For $k,\ell\in[1,n]$,

        \begin{enumerate}

            \setlength{\itemsep}{5pt}

            \item \label{X_eta} $[X, \eta_{k\ell}] =
                \left(\floor*{\frac{n-k}{i}} -
                \floor*{\frac{n-\ell}{i}}\right)\eta_{k,\ell+i} +
                \left(\floor*{\frac{n-k}{i}} + 1\right)e_{k-i, \ell}$

            \item \label{X_e} $[X, e_{k\ell}] = \left(\floor*{\frac{n - k}{i}}
                + 1 \right) e_{k-i,\ell} - \floor*{\frac{n-\ell}{i}} e_{k,\ell
                + i}$

            \item \label{Z_eta} $[Z, \eta_{k\ell}] = a_k\eta_{k+1,\ell} -
                a_{\ell-1} \eta_{k,\ell - 1} + \delta_{(n - \ell)
                \text{ mod } i, i - 1} a_{\ell-1} e_{n + k - \ell +
                1, n}$

            \item \label{Z_e} $[Z, e_{k\ell}] = a_k e_{k + 1, \ell} -
                a_{\ell-1} e_{k,\ell-1}$

        \end{enumerate}

    \item For $k\in[1, i-1]$ and $\ell\in[0, i]$,

            \begin{enumerate}

                   \setlength{\itemsep}{5pt}

                   \item \label{Z_E} $[Z, E_{k\ell}] = (i - 1 - k)E_{k+1,\ell}
                       - (-\ell \text{ mod } i)E_{k,\ell-1}$,

                   \item \label{X_E} $[X, E_{k\ell}] = -\delta_{\ell 0}E_{ki}$,

            \end{enumerate}

    \item For $\ell\in[1,i]$,

        \begin{enumerate}

            \setlength{\itemsep}{5pt}

            \item \label{X_d} $[X, d_\ell] = 0$,

            \item \label{Z_d} $[Z, d_\ell] = -(i-\ell)E_{\ell, \ell - 1}$.

        \end{enumerate}

\end{enumerate}

\end{lemma}

\begin{proof}

    For convenience, we first recall the relevant definitions: $\xi_{k\ell} :=
    \sum_{p\in\mathbb{Z}} e_{k+ip, \ell+ip}$, $\eta_{k\ell} := \sum_{p\geq 0}
    e_{k+ip, \ell+ip}$, $X := \sum_{1\leq j\leq n - i}c_je_{j,j+i}$, $Z :=
    \sum_{1\leq j < i}a_j\xi_{j+1,j}$, where $c_j = \floor*{\frac{n-j}{i}}$,
    $a_j := \epsilon((-nj) \text{ mod }  i)$ (where $\epsilon = 1$ if $n$ mod
    $i = 1$ and $\epsilon = -1$ if $n$ mod $i \neq 1$). Recall also $E_{k\ell}
    := \xi_{k+1, \ell+1}$ if $n$ mod $i = 1$ and $E_{k\ell} := -\xi_{i-\ell,
    i-k}$ if $n$ mod $i \neq 1$. Finally, $d_\ell := \sum_{\ell\leq j <
    i}E_{jj} + \varTheta \mathbf{I}_n$, where the scalar
    $\varTheta\in\mathbb{F}$ is uniquely determined by the condition that
    $d_\ell$ is traceless.

    The identities \ref{X_eta} and \ref{Z_eta} can be proved directly from
    these definitions. For instance, we compute

    \begin{align*}
        \left[X, \eta_{k\ell}\right] &= \left[ \sum_{1\leq j \leq n-i}
        c_j e_{j, j + i}, \sum_{p\geq 0} e_{k + ip, \ell + ip}\right] \\
        &= \sum_{\bgroup\def\arraystretch{0.5}
                    \begin{array}{c}
                        \scriptscriptstyle{1\leq j\leq n - i} \\
                        \scriptscriptstyle{p\geq 0}
                    \end{array}
                \egroup}
        c_j\left(\delta_{j+i,k+ip} e_{j, \ell+ip} - \delta_{j,\ell + ip}
        e_{k+ip, j+i}\right)
        \\
        &= \sum_{p\geq -1} c_{k + ip} e_{k + ip, \ell + (p + 1)i} -\sum_{p\geq
        0} c_{\ell +ip} e_{k + ip, \ell + (p + 1)i}
        \\
        &= c_{k-i} e_{k-i, \ell} + \sum_{p\geq 0} \left( c_{k+ip} - c_{\ell
        +ip}\right) e_{k + ip, \ell +(p+1)i}
        \\
        &= \left(c_{k} + 1\right) e_{k-i, \ell} + \sum_{p\geq 0} \left( c_k -
        c_\ell\right) e_{k + ip, \ell +(p+1)i}
        \\
        &= \left(c_{k} + 1\right) e_{k-i, \ell} + \left( c_k -
        c_\ell\right)\eta_{k, \ell + i}
    \end{align*}

    and

    \begin{align*}
        \left[Z, \eta_{k\ell}\right] &= \sum_{1\leq j<i}
        a_j\left[\xi_{j+1,j},\eta_{k\ell}\right]
        \\
        &= \sum_{\bgroup\def\arraystretch{0.5}
                    \begin{array}{c}
                        \scriptscriptstyle{1\leq j < i} \\
                        \scriptscriptstyle{p,q\geq 0}
                    \end{array}
                \egroup}
                a_j\left[e_{j+1+pi,j+pi},e_{k+qi, \ell + qi}\right]
        \\
        &= \sum_{\bgroup\def\arraystretch{0.5}
                    \begin{array}{c}
                        \scriptscriptstyle{1\leq j < i} \\
                        \scriptscriptstyle{p,q\geq 0}
                    \end{array}
                \egroup}
                a_j \left(\delta_{j+pi,k+qi}e_{k+1+qi,\ell+qi} -
                \delta_{j+1+pi,\ell + qi}e_{k+qi, \ell - 1 + qi}\right)
        \\
        &= a_k\eta_{k+1,\ell} - a_{\ell - 1}\eta_{k, \ell-1} + \delta_{(n-\ell)
            \text{ mod } i, i - 1}a_{\ell - 1} e_{n + k - \ell + 1, n}.
    \end{align*}

    Similarly, one can use the definitions provided above to prove identities
    \ref{X_e}, \ref{Z_e}, \ref{X_E}, and \ref{X_d}. To prove identities
    \ref{Z_E} and \ref{Z_d}, one can use an alternative formula for $Z$, namely
    $Z = \sum_{1\leq j<i} (i - j)E_{j,j-1}$ together with the observation that
    $[E_{k\ell}, E_{rs}] = \delta_{\ell r} E_{ks} - \delta_{sk} E_{r\ell}$ when
    $k,r\in [1, i-1]$ and $\ell,s\in [0, i]$.

\end{proof}

Since $\alpha_{\rightarrow}$, $\alpha_{\leftarrow}$, $\beta$, and $\gamma$ are
composed of terms only involving $e_{k\ell}$, $\eta_{k\ell}$, $E_{k\ell}$, and
$d_\ell$ with subscripts $k$ and $\ell$ in the ranges given in Lemma \ref{lemma
1}, we can readily compute the adjoint actions of $X$ and $Z$ on these
constituent parts of $r$ to obtain the following lemma.

\begin{lemma}

    \label{lemma 2}

    $ $

    \begin{enumerate}

    \setlength{\itemsep}{10pt}

        \item \label{X_alp_forward}

            $X. \alpha_{\rightarrow} = \displaystyle{-2\hspace{-4mm}\sum_{
                \bgroup\def\arraystretch{0.5}
                    \begin{array}{c}
                        \scriptscriptstyle{1\leq k<\ell\leq n} \\
                        \scriptscriptstyle{c_k > c_\ell}
                    \end{array}
                \egroup}\hspace{-4mm}
                e_{k\ell}\wedge \eta_{\ell,k+i}+2\sum_{1\leq j\leq n-i}
                \left(\sum_{j < s \leq j + i} \floor*{\frac{n-j}{i}}e_{js}
                \wedge e_{s,j+i}\right)}$

        \item \label{X_alp_backward}

            $X. \alpha_{\leftarrow} = \displaystyle{-2\sum_{1\leq k < \ell < i}
            E_{k\ell} \wedge E_{\ell - k, i}}$

        \item \label{X_beta}

            $X. \beta = \displaystyle{\sum_{1\leq j\leq
            n}\floor*{\frac{n-j}{i}} e_{j,j+i}\wedge\left(\frac{2}{n}I - e_{jj}
            - e_{j+i,j+i}\right)}$

        \item \label{X_gamma}

            $X. \gamma = \displaystyle{\sum_{1\leq j\leq
            n}\floor*{\frac{n-j}{i}} \left(e_{j,j+i}\wedge
            \left(e_{jj}-e_{j+i,j+i}\right) - 2\sum_{j < s < j+i}e_{js} \wedge
            e_{s,j+i}\right)}$

        \item \label{Z_alp_forward}

            $Z. \alpha_{\rightarrow} = \displaystyle{2\sum_{1\leq j <
            n}a_j\eta_{j+i+1,j+i}\wedge h_j}$

        \item \label{Z_alp_backward}

            $Z. \alpha_{\leftarrow} = \displaystyle{2\left(\sum_{1\leq
            j<i}(i-j)E_{j,j-1} \wedge E_{jj}+ \hspace{-2mm}
            \sum_{1\leq\ell<j\leq k\leq i} \hspace{-2mm} E_{j-1,k-1}\wedge
            E_{k-j+\ell,\ell-1} \right)}$

        \item \label{Z_beta}

            $Z. \beta = \displaystyle{-2\hspace{-1mm} \sum_{1\leq
            j<i}\hspace{-2mm} (i-j) E_{j,j-1} \hspace{-1mm} \wedge
            \hspace{-1mm} \left(d_j - d_{j + 1}\right) - \hspace{-2mm}
            \sum_{1\leq\ell<n} \hspace{-2mm} a_\ell\left(\eta_{\ell + 1, \ell}
            + \eta_{\ell + i + 1, \ell + i}\right) \wedge h_\ell}$

        \item \label{Z_gamma}

            $Z. \gamma = \displaystyle{\sum_{1\leq j<n}a_je_{j + 1, j} \wedge
            h_j}$

    \end{enumerate}

\end{lemma}

\begin{proof}

    Part \ref{X_alp_forward} follows from identities \ref{X_eta} and \ref{X_e}
    of Lemma \ref{lemma 1}. Part 2 follows from \ref{X_E} of Lemma \ref{lemma
    1}. Parts 3 and 4 follow from \ref{X_e}. Part 5 follows from \ref{Z_eta}
    and \ref{Z_e}. Part 6 follows from \ref{Z_E}. Parts 7 and 8 follow from
    \ref{Z_e}.

\end{proof}

\begin{proof}[Proof of Proposition \ref{module_actions}]

    Part \ref{Zr} follows directly from the identities \ref{Z_alp_forward},
    \ref{Z_alp_backward}, \ref{Z_beta}, and \ref{Z_gamma} of Lemma \ref{lemma
    2}. Part \ref{Xr} follows directly from the identities \ref{X_alp_forward},
    \ref{X_alp_backward}, \ref{X_beta}, and \ref{X_gamma} of Lemma \ref{lemma
    2}. For parts \ref{ZW} and \ref{ZV} (in the case when $\ell\neq 0$) use the
    identities \ref{Z_E} and \ref{Z_d} of Lemma \ref{lemma 1}.  To prove part
    \ref{ZV} in the case when $\ell  = 0$, we use the identities \ref{Z_eta}
    and \ref{Z_e} of Lemma \ref{lemma 1} to directly obtain

    \[
        Z.\left(\sum_{
            \bgroup
                \def\arraystretch{0.5}
                \begin{array}{c}
                    \scriptscriptstyle{1\leq k<\ell\leq n}
                    \\
                    \scriptscriptstyle{c_k > c_\ell}
                \end{array}
            \egroup}
            e_{k\ell}\wedge\eta_{\ell,k+i}\right)
        =
            \sum_{
                \bgroup
                \def\arraystretch{0.5}
                \begin{array}{c}
                    \scriptscriptstyle{1\leq k<\ell\leq n}
                    \\
                    \scriptscriptstyle{c_k > c_\ell}
                \end{array}
            \egroup}
            A_{k\ell} + B_{k\ell} + C_{k\ell},
        \]

    \noindent where $A_{k\ell} = a_k e_{k+1,\ell}\wedge\eta_{\ell,k+i} -
    a_{k-1}e_{k\ell}\wedge\eta_{\ell,(k-1)+i}$, $B_{k\ell} = a_\ell
    e_{k\ell}\wedge\eta_{\ell+1,k+i} -
    a_{\ell-1}e_{k,\ell-1}\wedge\eta_{\ell,k+i}$, and $C_{k\ell} =
    \delta_{(n-k) \text{ mod }i,i-1}a_{k-1}e_{k\ell} \wedge
    e_{n+1+\ell-k-i,n}$.  We observe that the $C_{k\ell}$'s all equal $0$
    because the subscript $n+1+\ell-k-i$ appearing in the expression for
    $C_{k\ell}$ is greater than $n$ whenever $(n-k) \text{ mod } i = i
    - 1$ and $c_k > c_\ell$.  Next we shift the indices of summation so that
    the terms involved in the $A_{k\ell}$'s and $B_{k\ell}$'s combine together
    into a single sum.  This gives us

    \begin{align*}
            Z. \left(\sum_{
                \hspace{-1mm}
                \bgroup
                    \def\arraystretch{0.5}
                    \begin{array}{c}
                        \scriptscriptstyle{1\leq k<\ell\leq n}
                        \\
                        \scriptscriptstyle{c_k > c_\ell}
                    \end{array}
                \egroup}
                \hspace{-4mm}
                e_{k\ell}\wedge\eta_{\ell,k+i}\right)
            &=
                a_n \left(
                \hspace{-1mm}
                \sum_{
                    \bgroup
                    \def\arraystretch{0.5}
                    \begin{array}{c}
                        \scriptscriptstyle{1\leq k<\ell\leq n}
                        \\
                        \scriptscriptstyle{c_k > c_{k+1} = c_\ell}
                    \end{array}
                \egroup}
                \hspace{-4mm}
                e_{k+1,\ell}\wedge\eta_{\ell,k+i} -
                \hspace{-4mm}
                \sum_{
                    \bgroup
                    \def\arraystretch{0.5}
                    \begin{array}{c}
                        \scriptscriptstyle{1\leq k\leq\ell\leq n}
                        \\
                        \scriptscriptstyle{c_k = c_\ell > c_{\ell+1}}
                    \end{array}
                \egroup}
                \hspace{-4mm}
                e_{k\ell}\wedge\eta_{\ell+1,k+i} \right)
            \\
            &=
                \epsilon a_n \sum_{1\leq k\leq i} E_{1k}\wedge E_{ki} =
                (i - 1) \sum_{1\leq k\leq i} E_{1k}\wedge E_{ki},
        \end{align*}

        \noindent Finally apply identity \ref{Z_E} of Lemma \ref{lemma 1} and
        it becomes a straightforward computation to verify that $Z. V_0 = V_1$.
        To prove parts \ref{XW} and \ref{XV}, use identities \ref{X_eta},
        \ref{X_e}, \ref{X_E}, and \ref{X_d} of Lemma \ref{lemma 1}.

\end{proof}

\subsection{The Principal Element \texorpdfstring{$\mathbf{H}$}{H}}
\label{subsection H}

Next we let $H\in\mathfrak{p}(i, n)$ denote the principal element corresponding
to the subprime functional (of Example \ref{principal_elt_subprime}) and let
$\alpha_0$ be the part of $\alpha$ obtained by adding the edge ${\mathcal
T}_{i,n}(n - (n \text{ mod } i)) =  i - (n \text{ mod } i)$ into the BD-triple
${\mathcal T}_{i, n}$.  Thus $\alpha = \alpha^\prime + \alpha_0$, where
$\alpha^\prime$ is the $\alpha$-part associated to the ``smaller'' BD-triple
identical to ${\mathcal T}_{i, n}$ except with the edge ${\mathcal T}_{i, n}(n
- (n \text{ mod } i)) = i - (n \text{ mod } i)$ removed. In fact, we observe
that $V_i = -\alpha_0$.  Put

\begin{equation}
    \label{definition of rprime}
    r^\prime := r - \alpha_0 = r + V_i \in M.
\end{equation}

\noindent Notice that $r^\prime$ is a solution to the MCYBE.

\begin{proposition}
    \label{H_M}

    The $\mathfrak{n}$-module $M$ decomposes into $H$-eigenspaces $M = \oplus
    M_\lambda$, where $M_\lambda = \{v\in M \mid H.v = \lambda v\}$. The
    nonzero eigenspaces $M_\lambda$ occur only if $\lambda\in[0, i + 1]$. Let
    $r^\prime$ be as given above. We have

    \begin{enumerate}

        \item $M_0$ is spanned by $r^\prime$,

        \item $M_1$ is spanned by $V_0$ and $W$,

        \item for $\lambda\in[2,i+1]$, $M_\lambda$ is spanned by
            $V_{\lambda - 1}$.

    \end{enumerate}

\end{proposition}

\begin{proof}

    First, compute to obtain the identity $[H, E_{k\ell}] = (k - \ell +
    \delta_{\ell i}(i + 1))E_{k\ell}$ for all $k\in[1,i-1]$ and
    $\ell\in[0,i]$.  We also have the identities $[H, X] = X$ and $[H,
    Z] = Z$.  Since the $\beta$- and $\gamma$-parts of $r$ are annihilated by
    the adjoint action of $H$ (in fact, $\beta$ and $\gamma$ are annihilated by
    the adjoint action of any diagonal matrix) and additionally
    $H.\alpha_{\rightarrow} = 0$, this implies $H.r = H.\alpha_{\leftarrow}$.
    As $\alpha_{\leftarrow}$ can be written in terms of the $E_{k\ell}$'s, we
    obtain

    \begin{align*}
        H.r &= H.\alpha_{\leftarrow}
        = 2\sum_{1\leq \ell \leq j <k \leq i}\delta_{ki}(i+1)
        E_{jk}\wedge E_{k - j + \ell - 1, \ell - 1}
        \\
        &= 2\sum_{1\leq \ell \leq j <i}(i+1)
        E_{ji}\wedge E_{i - j + \ell - 1, \ell - 1}
        = -(i+1)V_i
    \end{align*}

    \noindent Now that we have established how $H$ acts on $X$, $Z$, and
    $r$, Proposition \ref{module_actions} provides a way to compute the
    adjoint action of $H$ on the $V_j$'s and $W$. We obtain $H.V_j = (j+1)V_j$
    (for $j \in [0, i]$) and $H.W = W$. Therefore $H.r^\prime =
    H(r + V_i) = H.r + H.V_i = 0$.

\end{proof}

\subsection{The map \texorpdfstring{$\mathbf{\Phi: \mathbb{F}^\times \to
SL_n}$}{Phi:$F^x$ --> SL(n)}}

\label{subsection phi}

Define the matrix

\begin{equation}
    \label{g_definition}
    g := e^{-Z}e^{-X}exp\left(-\sum_{1\leq k < i} \binom{i}{k}
    E_{ki}\right) \in SL_n.
\end{equation}

The matrix $g\in SL_n$ is obtained by seeking a matrix of the form $e^{\bf X}$
with ${\bf X}\in\mathfrak{n}$ so that $e^{\bf X}.r\in M_0\oplus M_1$. These
conditions significantly restrict plausible candidates for $g$. They, in fact,
uniquely determine $g$ up to a single non-zero free parameter in
$\mathbb{F}^\times$.  After obtaining such a matrix $g$ ($=e^{\bf X}$), we
observe that it can be factored into the product form as we have defined it
above. The factored form of $g$ highlights the roles of the matrices $X$ and
$Z$ and allows us to easily prove the following proposition.

\begin{proposition}

    \label{g has integer entries}

    All entries in the matrix $g$ are integers.

\end{proposition}

\begin{proof}

    From the definitions of the matrices $X$ and $Z$, it is easy to see that
    for every $k >0$, all entries in $X^k$ and $Z^k$ are products of $k$
    consecutive integers. Hence the matrices $X^k / k!$ and $Z^k / k!$  have
    integer entries. Therefore $e^X$ and $e^Z$ have integer entries.  Next, we
    note that $exp\left(-\sum_{1\leq k <i} \binom{i}{k} E_{ki}\right) = I -
    \sum_{1\leq k <i} \binom{i}{k} E_{ki}$. Thus, the matrix
    $exp\left(-\sum_{1\leq k <i} \binom{i}{k} E_{ki}\right)$ has integer
    entries. Since $g$ is written as the product of three matrices, each of
    them having integer entries, $g$ also has integer entries.

\end{proof}

We observe that multiplying the principal element $H$ by $2n$ gives us a
diagonal matrix having \textit{integer} entries. This implies that $t^{2nH}$ is
a diagonal matrix having \textit{polynomial} entries in $t^{\pm 1}$.  In fact,
$t^{2nH} = \operatorname{diag}(t^{2nH_1}, t^{2nH_2},\dots, t^{2nH_n})$ whenever
$H = \operatorname{diag}(H_1, H_2,\dots, H_n)$.  By rescaling we avoid having
to consider whether or not roots of $t$ exist. For instance $t^{1/j}$ may not
exist for every $t\in\mathbb{F}$ if $\mathbb{F}$ is not algebraically closed.
In short, rescaling is the reason why we do not need to require that
$\mathbb{F}$ be algebraically closed.

Finally, define the map $\Phi: \mathbb{F}^\times\to SL_n$ by $\Phi(t) :=
t^{2nH}g$.  Since $t^{2nH}$ is a diagonal matrix with nonzero entries of the
form $t^m$ with $m$ an integer, Proposition \ref{g has integer entries} implies
that all entries in $\Phi(t)$ have the form $ct^p$ with $c$ and $p$ integers.

\section{Main Result}

\label{main result section}

We are now in a position to prove the main result. Define

\begin{equation}
    \label{definition of b}
    b := -(X + Z)r = -V_0 - W \in \mathfrak{p}(i,n) \wedge \mathfrak{p}(i, n).
\end{equation}

\begin{theorem}

    \label{main_result}

    Suppose $n\equiv \pm 1$ (mod $i$) and let $r = r_{CG}(i, n)$ be the
    corresponding Cremmer-Gervais $r$-matrix of type $(i, n)$.  Let $r^\prime$,
    $b$, and $\Phi(t)$ be as defined above, then

    \begin{enumerate}

        \item $\Phi(t).r = r^\prime + t^{2n}\cdot b$.

        \item $b$ is a boundary solution to the CYBE having carrier
            $\mathfrak{p}(i, n)$

        \item $b$ lies in the closure of the component of ${\mathcal M}$
            containing the Cremmer-Gervais $r$-matrix $r$

    \end{enumerate}

\end{theorem}

\begin{proof}

    To prove part 1, observe first that for $k\in[1,i-1]$ we have the
    identity $(\operatorname{ad}Z)^kX = \frac{(i-1)!}{(i-k-1)!}E_{ki}$. Thus we
    can use the identities in Proposition \ref{module_actions} to compute the
    adjoint action of $exp(-\sum_{1\leq k < i} \binom{i}{k} E_{ki})$ on $r$.
    We compute

    \begin{align*}
        g.r
        &=
        e^{-Z}e^{-X}.\left(r - \sum_{1\leq k < i}V_k\right)
        =
        e^{-Z}.\left(r - \sum_{0\leq k < i} V_k\right)
        \\
        &=
        r - W - V_0 + V_i = r^\prime + b
    \end{align*}

    \noindent  Since $t^{2nH}v = t^{2n\lambda}v$ for every $H$-eigenvector
    $v\in M_\lambda$ and we have, in particular, that $r^\prime \in M_0$ and
    $b\in M_1 $, it follows that $t^{2nH}g.r = t^{2nH} (r^\prime + b) =
    r^\prime + t^{2n}b$. Part 3 follows as a direct consequence of part 1.  For
    part 2 we note that $\langle r, r\rangle = \langle r^\prime, r^\prime
    \rangle$, therefore Theorem \ref{constructing boundary solutions, 1}
    implies that $b$ is a boundary solution to the CYBE. The definitions of
    $V_0$ and $W$ show that the carrier of $b$ is $\mathfrak{p}(i, n)$.

\end{proof}

Part 2 of Theorem \ref{main_result} can be slightly generalized. In fact, for
all scalars $\rho,\mu\in \mathbb{F}$, $b_{\rho,\mu} := (\rho X+\mu Z).r$ is a
boundary solution to the CYBE. Furthermore if $\rho$ and $\mu$ are both
nonzero, the carrier of $b_{\rho,\mu}$ is the maximal parabolic subalgebra
$\mathfrak{p}(i, n)\subseteq \mathfrak{sl}_n$ and $b_{\rho,\mu}$ admits the
Frobenius functional

\begin{equation}
    f_{\rho, \mu} := -\rho^{-1}\left(\sum_{i<j\leq n}e_{j-i,j}^*\right)
    +\epsilon\mu^{-1} \left(\sum_{1\leq j <
    i}e_{j+1,j}^*\right)\in\mathfrak{p}(i, n)^*.
\end{equation}

\noindent The principal element corresponding to $f_{\rho, \mu}$ is equal to
$H$ (from Example \ref{principal_elt_subprime}), independent of the scalars
$\rho$ and $\mu$. Notice that when the scalars are specialized to $\rho = -1$
and $\mu = \epsilon$, the functional $f_{\rho, \mu}$ reduces to the subprime
functional.

\section{Closing Remarks}

\label{closing}

The results in the previous sections prove all cases of the
Gerstenhaber-Giaquinto boundary conjecture when $i = 1,2,3,$ and $4$. The next
smallest case to consider is when $i=5$ and $n=7$. However, this particular
case is handled by applying the non-trivial Dynkin diagram automorphism of
$\mathfrak{sl}_7$ to the $(2, 7)$-case. Similarly, the case when $i=5$ and
$n=8$ is handled by applying Dynkin diagram automorphism to the $(3, 8)$-case.
After these cases, the next smallest case to consider is when $i=5$ and $n=12$.

Since the subprime functional $f\in \mathfrak{p}(i, n)^*$ is Frobenius if and
only if $n\equiv \pm 1$ (mod $i$) it seems unlikely that $f$ will play as
important a role, if any, in proving the Gerstenhaber-Giaquinto boundary
conjecture in cases when $n\not\equiv\pm 1$ (mod $i$).  The subprime functional
is, first of all, an example of a \textit{small} functional. This means that if
we write the subprime functional as $f = \sum_{(j,k)\in E}e_{jk}^*$, with $E =
\{(1, i+1), (2,i+2), \dots, (n-i,n), (2,1), (3,2), \dots, (i+1, i)\}$, the
corresponding directed graph having vertices $1,\dots,n$ and edges $j\to k$ for
each $(j,k)\in E$ has an underlying undirected graph a tree. In fact any
functional $f$ of the form $\sum_{(j,k)\in S}e_{jk}^*$, where $S$ is the
\textit{edge set}, with corresponding graph a tree is called a \textit{small}
functional (see e.g.  \cite[Section 3]{GG2}).  The edge set $S$ is the
\textit{support} of $f$. The smallest possible size for the support of a
Frobenius functional $f\in \mathfrak{p}(i, n)^*$ is $n - 1$.

In proving the Gerstenhaber-Giaquinto boundary conjecture for the $(5, 12)$
case it seems that we must look beyond the small functionals.  To be more
precise, first recall that for any Frobenius Lie algebra $\mathfrak{f}$ and
Frobenius functional $f\in\mathfrak{f}^*$ we can produce a solution to the
classical Yang-Baxter equation $r_f\in {\mathcal C}$ having carrier
$\mathfrak{f}$ by inverting the matrix $(B_{jk})$, where $B_{jk} =
f([x_j,x_k])$ with respect to a basis $x_1,\dots,x_d$ of $\mathfrak{f}$. We
simply put $r_f = \sum (B_{jk})^{-1}x_j\wedge x_k$.  In this paper, part of the
main result illustrates that when $f \in \mathfrak{p}(i, n)^*$ is the subprime
functional and $n\equiv \pm 1$ (mod $i$), there exists  $x\in\mathfrak{sl}_n$
so that $x.r_{CG}(i, n) = r_f$, namely $x = -X +\epsilon Z$. In view of this,
one approach to proving the boundary conjecture in the $(5, 12)$-case would be
to find $x\in\mathfrak{sl}_{12}$ and a Frobenius functional
$f\in\mathfrak{p}(5, 12)^*$ so that $x.r_{CG}(5, 12) = r_f$.

With this in mind, consider the $12\times 12$ matrix $x = (a_{jk})$ where the
entry in the $(j,k)$-position is a variable named $a_{jk}$. Thus, with a fixed
functional $f$, the equation $x.r_{CG}(5, 12) = r_f$ represents a linear system
of $12^4$ equations in $144$ variables.  However computer calculations indicate
that this system is inconsistent for each small Frobenius functional $f\in
\mathfrak{p}(5, 12)^*$.

Rather than searching first for a plausible functional $f$ that solves the
system $x.r_{CG}(5, 12) = r_f$, we could instead first find $x \in
\mathfrak{sl}_{12}$ so that $x.r_{CG}(5, 12)\in {\mathcal C}$.  However this
amounts to solving a system of $144^3$ quadratic equations in $144$ variables,
which is a difficult task.  Instead we only look for $x$ of the form
$\sum_{jk}a_{jk}e_{jk}$ with the scalars $a_{jk} \in \mathbb{F}$ nonzero only
when $j - k$ belongs to a sufficiently ``small'' subset of integers. This
restriction significantly reduces the number of equations and variables of the
quadratic system. This approach likely leads to several plausible candidates
for $x$, but we only want to consider those with $x.r_{CG}(5, 12)$ having
carrier equal to $\mathfrak{p}(5, 12)$. For example, define

\begin{equation}
    x := X_7 + X_4 + X_1 + X_{-2} \in \mathfrak{p}(5, 12),
\end{equation}

\noindent where $X_7$, $X_4$, $X_1$, $X_{-2}\in\mathfrak{p}(5, 12)$ are defined
as

\begin{align}
    &X_7 := -\sum_{1\leq j \leq 5} e_{j, j+7},
    &
    &X_1 := -\left(e_{1,2} + e_{3, 4} + 2e_{5,6} + e_{6,7} + e_{8,9}
    +e_{10,11}\right),
    \\
    &X_4 := -e_{2,6} - e_{3,7}\in,
    &
    &X_{-2} := \sum_{1\leq j\leq 5}\floor*{\frac{7-j}{2}}\left(e_{j,j-2} +
    e_{j+7,j+5}\right),
\end{align}

\noindent and let $r = r_{CG}(5, 12)$ be the Cremmer-Gervais $r$-matrix of type
$(5, 12)$. Computer calculations verify that $x.r$ is a solution to the CYBE
having carrier $\mathfrak{p}(5, 12)$ and admits the Frobenius functional

\begin{equation}
    f = \left(\sum_{1\leq j\leq 5}e_{j,j+7}^*\right) + \left(\sum_{1\leq j\leq
    5}e_{j+7, j+5}^*\right) + e_{6,7}^* + e_{7,8}^* + e_{6,10}^* + e_{7,11}^*
    \in\mathfrak{p}(5, 12)^*.
\end{equation}

\noindent Observe that $f$ is not a small Frobenius functional because its
support has size greater than $11$. The principal element associated to $f$ is

\begin{equation}
    H = \frac{1}{2}\cdot \operatorname{diag} \left(1, -1, 3, 1, 5, -3, -5, -1,
    -3, 1, -1, 3\right) + 3\left(e_{3,6} + e_{4,7} + e_{1,7}\right)\in
    \mathfrak{p}(5, 12).
\end{equation}

Similar to our previous constructions, we now let $\mathfrak{n}$ denote the Lie
subalgebra of $\mathfrak{sl}_{12}$ generated by the matrices $X_7$, $X_4$,
$X_1$, and $X_{-2}$.  The $\mathfrak{n}$-module $M:={\mathcal
U}(\mathfrak{n}).r$ decomposes into $H$-eigenspaces $M = M_0\oplus M_1\oplus
\cdots \oplus M_9$, where $M_\lambda := \{v\in M \mid H.v = \lambda v \}$,
having respective dimensions $1, 4, 5, 9, 10, 12, 10, 8, 4, 2$. Interestingly,
the subspace $M_0$ is spanned by

\begin{equation}
    r^\prime := exp\left(-e_{3, 6} - e_{4, 7} - e_{1, 7}\right).\left(r -
    \alpha_0\right),
\end{equation}

\noindent where $\alpha_0$ is the $\alpha$-part of $r$ obtained by adding the
edges $5 \to 10$ and $6 \to 11$ to the Cremmer-Gervais graph $\Gamma_{CG}(5,
12)$. We note that $r^\prime$ is a solution to the MCYBE equivalent to
$r-\alpha_0$.

We also have $x.r\in M_1$. Thus $t^{2H}(x.r) = t^2(x.r)$. We remark that
$t^{2H}$ has polynomial entries in $t^{\pm 1}$. The matrix $t^{2H}$ can be
obtained by diagonalization: $2H = PDP^{-1}$ with $D$ a diagonal matrix having
integer entries $D = \operatorname{diag}(H_1, H_2,\dots, H_{12})$. Next put
$t^{2H} = P \operatorname{diag} (t^{2H_1}, \dots, t^{2H_{12}}) P^{-1}$.  Define
$Y := -2e_{1,6} - e_{2,6} + 3e_{3,6} + 2e_{4,6} + e_{5,6} - e_{3, 7} - e_{4, 7}
- e_{5, 7} \in \mathfrak{n}$ and put

\begin{equation}
    g := e^{X_7}e^{X_{-2}}e^{[X_{-2}, X_1]}e^{X_1}e^Y \in SL_{12}.
\end{equation}

As before, the matrix $g$ is obtained by seeking a matrix of the form $g=e^{\bf
X}$ with ${\bf X}\in\mathfrak{n}$ so that $e^{\bf X}.r\in M_0\oplus M_1$ and
observing that $g$ factors into the form written above.  We have $g.r =
r^\prime + x.r$. Therefore $t^{2H}g.r = r^\prime + t^2x.r$. Since
$\langle t^{2H}g.r, t^{2H}g.r\rangle = \langle r^\prime, r^\prime\rangle$,
Theorem \ref{constructing boundary solutions, 1} implies $x.r$ is a boundary
solution to the CYBE lying in the closure of the Belavin-Drinfeld component of
${\mathcal M}$ containing $r$.

\end{document}